\documentclass[leqno, twoside]{amsart}
\usepackage{amsthm}
\usepackage{amsmath}
\usepackage{amssymb}
\usepackage{graphicx}
\usepackage{epstopdf}
\usepackage{epsfig}

\font\sn = cmssi8 scaled \magstep0

\newif\ifdraft\drafttrue
\draftfalse
\newcommand\name[1]{\label{#1}{\ifdraft{\sn [#1]}\else\ignorespaces\fi}}

\newcommand\eq[2]{{\ifdraft{\ \tt [#1]}\else\ignorespaces\fi}\begin{equation}\label{#1}{#2}\end{equation}}
\newcommand {\equ}[1]{\eqref{#1}}

\newcommand{\Q}{{\mathbb {Q}}}

\newcommand{\R}{{\mathbb{R}}}

\newcommand{\N}{{\mathbb{N}}}

\newcommand{\GL}{\operatorname{GL}}

\newcommand{\SL}{\operatorname{SL}}

\newcommand {\ignore}[1]  {}

\newcommand{\sm}{\smallsetminus}

\newcommand{\Aff}{{\mathrm{Aff}}}

\title{Parking garages with optimal dynamics}
\author{Meital Cohen}
\address{Ben Gurion University, Be'er Sheva, Israel 84105
{\tt comei@bgu.ac.il}}
\author{Barak Weiss}
\address{Ben Gurion University, Be'er Sheva, Israel 84105
{\tt barakw@math.bgu.ac.il}}

\newtheorem{thm}{Theorem}[section]
\newtheorem*{Thm}{Theorem}
\theoremstyle{definition}
\newtheorem{definition}[thm]{Definition}
\newtheorem*{Definition}{Definition}

\theoremstyle{plain}
\newtheorem{lem}[thm]{Lemma}
\newtheorem{prop}[thm]{Proposition}
\newtheorem{cor}[thm]{Corollary}

\newtheorem{remark}[thm]{Remark}
\newtheorem*{Remark}{Remark}

\theoremstyle{remark}

\begin{document}

\maketitle
\begin{abstract}
We construct generalized polygons (`parking garages') in which the
billiard flow satisfies the Veech dichotomy, although the associated
translation surface obtained from the Zemlyakov-Katok unfolding is not
a lattice surface. We also explain the difficulties in constructing
a genuine polygon with these properties. 
\end{abstract}

\section{Introduction and Statement of results}
A {\em parking garage} is an immersion $h: N \to \R^2$, where $N$ is a
two dimensional compact connected manifold 
with boundary, and $h(\partial N)$ is
a finite union of linear segments. A parking garage is called {\em
  rational} if the group generated by the linear 
parts of the reflections in the boundary segments is finite. 
If $h$ is
actually an embedding, the 
parking garage is a polygon; thus polygons form a subset of parking
garages, and rationals polygons (i.e. polygons all of whose angles are
rational multiples of $\pi$) form a subset of rational parking garages.

The dynamics of the billiard flow in a rational polygon has been intensively
studied for over a century; see \cite{FK} for an early example, 
and \cite{DeMarco, MT, Vorobets, Zorich} for recent surveys. The
definition of the 
billiard flow on a polygon readily extends to a parking garage: 
on the interior of $N$ the billiard flow is the geodesic flow on
the unit tangent bundle of $N$ (with respect to the pullback of the
Euclidean metric) and at the boundary, the flow is
defined by elastic reflection (angle of incidence equals the angle of
return). The flow is undefined at the finitely many points of $N$ which
map to `corners', i.e. endpoints of boundary segments, and hence at the countable
union of codimension 1 submanifolds corresponding to points in the
unit tangent bundle for which the corresponding geodesics eventually
arrive at corners in positive or negative time. Since the direction of
motion of a trajectory changes at a boundary segment via a reflection in its
side, for rational parking garages, only finitely many directions of
motion are assumed. In other words, the phase space of the billiard
flow decomposes into invariant two-dimensional subsets corresponding
to fixing the directions of motion.

Veech \cite{Veech - dichotomy} discovered that the billiard flow in
some special polygons exhibits a striking
dichotomy. Namely he found polygons for which, in any initial
direction, the flow is either 
{\em completely periodic} (all orbits are periodic), or {\em uniquely
  ergodic} (all orbits are equidistributed). Following McMullen
we will say that a polygon with these properties has
     {\em optimal dynamics}. We briefly summarize Veech's strategy of
     proof. A standard unfolding construction usually attributed to
     Zemlyakov and Katok \cite{ZK}\footnote{but dating back at least to Fox
     and Kershner \cite{FK}.}, 
     associates to any rational polygon $\mathcal{P}$ a translation
     surface $M_{\mathcal{P}}$,
such that the billiard flow on $\mathcal{P}$ is
     essentially equivalent to the straightline flow on
     $M_{\mathcal{P}}$. Associated with any translation surface $M$ is a
     Fuchsian group $\Gamma_M$, now known as the {\em Veech group} of
     $M$, which is typically trivial. Veech found 
$M$ and $\mathcal{P}$ for which this group is a lattice in
     $\SL_2(\R)$. We will call these 
     {\em lattice
       surfaces} and {\em lattice polygons} respectively. Veech investigated the 
     $\SL_2(\R)$-action on the moduli space of translation surfaces,
     and building on earlier work of Masur, showed that lattice
     surfaces have optimal dynamics. From this it follows that {\em
       lattice polygons have optimal dynamics.}

This chain of reasoning remains valid if one starts with a parking garage
instead of a polygon; namely, the unfolding construction
associates a translation surface to a parking garage, and one may
define a lattice parking garage in an analogous way. The arguments of
Veech then show that the billiard flow in a lattice parking garage has
optimal dynamics. This generalization is not vacuous: lattice parking
garages, which are not polygons, were recently discovered by Bouw and
M\"oller \cite{BM}. The term `parking garage' was coined by M\"oller.

A natural question is whether Veech's
result admits a converse, i.e. whether non-lattice polygons or parking
garages may also
have optimal dynamics. In \cite{SW}, Smillie and the second-named
author showed that there are non-lattice translation surfaces 
which have optimal dynamics. However translation surfaces arising from
billiards form a set of measure
zero in the moduli space of translation surfaces, and it was not clear
whether the examples of \cite{SW} arise from polygons or parking
garages. In this paper we show:

\begin{thm}\name{thm: intro}
There are non-lattice parking garages with optimal dynamics. 

\end{thm}

An example of such a parking garage is shown in Figure \ref{parking}. 
\begin{figure}[h!]
\begin{center}
\includegraphics[scale=0.7]{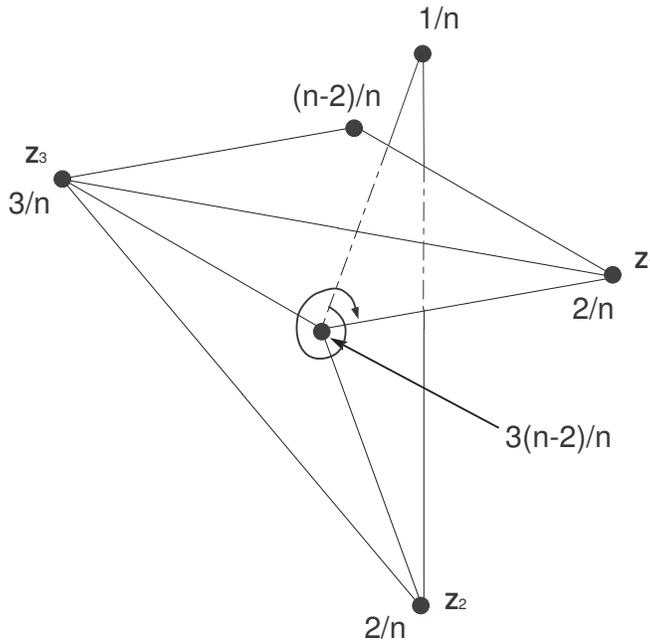}
\caption{
An non-lattice parking garage with optimal dynamics.}
\label{parking}			
\end{center}
\end{figure}

Veech's work shows that for lattice polygons, the directions in which
all orbits are periodic are precisely those containing a {\em saddle connection},
i.e. a billiard path
connecting corners of the 
polygon which unfold to singularities of the corresponding
surface. Following Cheung, Hubert and Masur \cite{CHM}, if a 
polygon $\mathcal{P}$ has optimal dynamics, and the periodic directions
coincide with the directions of saddle connections, we will say
that $\mathcal{P}$ satisfies {\em 
  strict ergodicity and topological dichotomy}. 
It is not clear to us whether our example satisfies this
stronger property. As we explain in Remark \ref{remark: connection points}
below, this would follow if it were 
known that the center of the regular $n$-gon is a `connection point' in
the sense of Gutkin, Hubert and Schmidt
\cite{GHS} for some $n$ which is an odd multiple of 3.  

Veech also showed that for a lattice polygon $\mathcal{P}$, the number
$N_{\mathcal{P}}(T)$ of periodic strips on $\mathcal{P}$ of
length at most $T$ satisfies a quadratic growth estimate of the form
$N_{\mathcal{P}}(T) \sim cT^2$ for a positive constant $c$. As we
explain in Remark \ref{remark: quadratic growth}, our examples
also satisfy such a quadratic growth estimate.

It remains an open question whether there is a genuine polygon which
has optimal dynamics and is not a lattice polygon. Although our
results make it seem likely that such a polygon exists, in her M.Sc. thesis
\cite{Meital thesis}, the first-named author obtained severe 
restrictions on such a polygon. In particular she showed that there are
no such polygons which may be constructed from any of the currently
known lattice examples 
via the covering construction as in \cite{Vorobets, SW}. We explain
these results and prove a representative special case in \S \ref{section: meital thesis}.

\subsection{Acknowledgements}
 We are grateful to Yitwah Cheung and
Patrick Hooper for helpful discussions. 
This research was supported by the Israel Science Foundation and the
Binational Science Foundation. 

\section{Preliminaries}
In this section we cite some results which we will need, and deduce
simple consequences. For the sake of brevity we will refer the reader
to \cite{MT, Zorich, SW} for definitions of translation surfaces.   

Suppose $S_1, S_2$ are compact orientable surfaces and $\pi: S_2 \to
S_1$ is a branched cover. That is, $\pi$ is continuous and surjective,
and there is a finite $\Sigma_1 \subset
S_1$, called the set of {\em branch points}, such that for $\Sigma_2 =
\pi^{-1}(\Sigma_1)$, the restriction of $\pi$ to $S_2 \sm \Sigma_2$ is
a covering map of finite degree $d$, and for any $p \in \Sigma_1$, 
$\# \pi^{-1}(p) < d$. A {\em ramification point} is a point $q \in
\Sigma_2$ for which there 
is a neighborhood $\mathcal{U}$ such that $\{q\} = \mathcal{U} \cap
\pi^{-1}(\pi(q))$ and for all $u \in \mathcal{U}
\sm \{q\}$, $\# \pi^{-1}(\pi(u)) \geq 2.$ 

If $M_1, M_2$ are translation surfaces, a {\em translation map} is a
surjective map $M_2 \to M_1$ which is a 
translation in charts. It is a branched cover. In contrast to other
authors (cf. \cite{GHS, Vorobets}), we do not require that the set of
branch points be distinct from the singularities of $M_1$, or that
they be marked. It is clear that the ramification points of the cover
are singularities on $M_2$. 

If $M$ is a lattice surface, a point $p \in M$ is called {\em
  periodic} if its orbit under the group of affine automorphisms of
$M$ is finite. A point $p \in M$ is called a {\em connection point} if
any segment joining a singularity with $p$ is contained in a saddle
connection (i.e. a segment joining singularities) on $M$. The
following proposition summarizes results discussed in \cite{FK, MT, SW, HS}:  
\begin{prop}\name{prop: basics}
\item[(a)]
A non-minimal direction on a translation surface contains a saddle
connection. 

\item[(b)]
If $M_1$ is a lattice surface, $M_2 \to M_1$ is translation map with a
unique branch point, then any minimal direction on $M_2$ is uniquely ergodic.

\item[(c)]
If $M_2 \to M_1$ is a translation map such that
$M_1$ is a lattice surface and the set of branch points contains a
non-periodic point, then $M_2$ is 
not a lattice surface. 

\item[(d)]
If $M_2 \to M_1$ is a translation map such that
$M_1$ is a lattice surface and all branch points are connection
points, then any saddle connection direction on $M_2$ is periodic. 
\end{prop}

\begin{cor}\name{cor: suitable cover}
Let $M_2 \to  M_1$ be a translation map such that $M_1$ is a lattice
surface with a unique branch point $p$. Then:
\begin{enumerate}
\item $M_2$ has
optimal dynamics. 

\item
If $p$ is a connection point then $M_2$ satisfied
topological dichotomy and strict ergodicity. 
\item
If $p$ is not a periodic point then $M_2$ is not a
lattice surface. 

\end{enumerate}
\end{cor}

\begin{proof}
To prove (1), by (b), the minimal directions are uniquely
ergodic, and we need to prove that the remaining directions are
either completely periodic or uniquely ergodic. 
By (a), in any non-minimal direction on $M_2$ there is a saddle
connection $\delta$, and there are 
three possibilities: 
\begin{enumerate}
\item[(i)] $\delta$ projects to a saddle connection on $M_1$.
\item[(ii)] $\delta$ projects to a geodesic segment connecting the
  branch point $p$ to itself.
\item[(iii)] $\delta$ projects to a geodesic segment connecting $p$ to a singularity.
\end{enumerate}

In case (i) and (ii) since $M_1$ is a lattice surface, the direction is
periodic on $M_1$, hence on $M_2$ as well. In case (iii), there are
two subcases: if $\delta$ projects to a part of a saddle
connection on $M_1$, then it is also a periodic direction. Otherwise, the
direction must be minimal in $M_1$, hence uniquely ergodic in
$M_2$. This proves (1). Note also that if $p$ is a connection point
then the last subcase does not arise, so all directions which are
non-minimal on $M_2$ are periodic. This proves (2). Statement (3)
follows from (c).
\end{proof}

We now describe the unfolding construction \cite{FK, ZK}, extended
to parking garages. 
Let $\mathcal{P} = (h: N \to \R^2)$. An {\em edge} of $\mathcal{P}$
is a connected subset $L$ of $\partial N$ such that $h(L)$ is a
straight segment and $L$ is maximal with these properties (with
respect to inclusion). A {\em vertex} of $\mathcal{P}$ is any point
which is an endpoint of an edge. The {\em angle} at a vertex is
the total interior angle, measured via the pullback of the Euclidean
metric, at the vertex. By convention we always choose the positive
angles. Note that for polygons, angles are less than
$2\pi$, but for parking garages there is no apriori upper bound on the
angle at a vertex. Since our parking garages are rational, all angles
are rational multiples of $\pi$, and we always write them as $p/q$,
omitting $\pi$ from the notation.

Let $G= \{g_1,
\ldots, g_r\}$ be the dihedral group generated by the linear parts of
reflections in the $h(L_j)$.
Let $S$ be the topological space obtained from $N \times \{1,
\ldots, r\}$ by identifying $(x,i)$ with $(x,j)$ whenever $g_i^{-1}
g_j$ is the linear part of the reflection in an edge containing 
$h(x)$. Topologically $S$ is a compact orientable surface, and the
immersions $g_i \circ h$ on each $N \times \{i\}$ induce an atlas of
charts to $\R^2$ which endows $S$ with a translation surface
structure. We denote this translation surface by 
$M_{\mathcal{P}}$. 

We will be interested in a `partial unfolding' which is a variant of this
construction, in which we reflect a parking garage repeatedly around
several of its edges to form a larger parking garage. Formally,
suppose $\mathcal{P} = (h: N \to \R^2)$ and  $\mathcal{Q} = (h': N'
\to \R^2)$ are parking garages. We say that $\mathcal{P}$ {\em tiles
  $\mathcal{Q}$ by reflections} if the following holds. Let $G$ be
the group generated by linear parts of reflections in the edges of
$\mathcal{P}$. Then $N'$ is tiled
by sets $N'_1, \ldots, N'_{\ell}$ homeomorphic to $N$, such that for
each $j$, $h'(N'_j)$ and $h(N)$ differ by an isometry of the plane
with linear part $g_j \in G$, and any connected component of $h'(N'_i
\cap N'_j)$ which is not a point is the image of an edge of
$\mathcal{P}$ under both $g_i$ and $g_j$, such that $g_i^{-1}\circ
g_j$ is the reflection in this edge. We call $\ell$ the {\em number of
  tiles}. 

Vorobets \cite{Vorobets} realized that a tiling of parking garages
gives rise to a branched cover. More precisely: 

\begin{prop}\name{prop: precisely}
Suppose $\mathcal{P}$ tiles $\mathcal{Q}$ by reflections with $\ell$ tiles, 
$M_{\mathcal{P}}, M_{\mathcal{Q}}$ are the corresponding translation surfaces
obtained via the unfolding construction, and $G_{\mathcal{P}},
G_{\mathcal{Q}}$ are the corresponding reflection groups. Then 
there is a translation map $M_{\mathcal{Q}} \to M_{\mathcal{P}}$, such
that the following hold:

\begin{enumerate}
	\item $G_{\mathcal{Q}} \subset G_{\mathcal{P}}$. 
\item  The branch points are contained in the $G_{\mathcal{P}}$-orbit
  of the vertices of $\mathcal{P}$.
	\item The degree of the cover is
          $\frac{\ell}{[G_{\mathcal{P}}: G_{\mathcal{Q}}]}.$


\item Let $z \in M_{\mathcal{P}}$ be a point which corresponds to a
  vertex in $\mathcal{P}$  with
  angle $\frac{m}{n}$ (where $\gcd (m,n)=1$), and let $(y_i) \subset
  M_{\mathcal{Q}}$ be 
the pre-images of $z$, with angles 
  $\frac{k_i\, m}{n}$ in $\mathcal{Q}$.  
Then $z$ is a branch point of the cover if and only
if $k_{i} \nmid n$ for some $i$. 
\end{enumerate}
\end{prop}

\begin{proof}
Assertions (1) and (2) are simple and left to the reader. For
assertion (3) we note that $M_{\mathcal{P}}$ (resp. $M_{\mathcal{Q}}$)
is made of $|G_{\mathcal{P}}|$ (resp. $\ell \, |G_{\mathcal{Q}}|$)
copies of $\mathcal{P}$.  
The point $z$ will be a branch point if and only if the total angle
around $z \in M_{\mathcal{P}}$ differs from
the total angle around one of the pre-images $y_i \in
M_{\mathcal{Q}}$. The total angle at a singularity corresponding to a
vertex with angle $r/s$ (where $\gcd(r,s)=1$) is $2r\pi$, thus the
total angle at $z$ is $2m\pi$ and the
total angle at  $y_i $ is 
$\frac{2k_im\pi}{\gcd(k_i,n)}$. Assertion (4) follows. 
\end{proof}

\section{Non-lattice dynamically optimal parking garages}
In this section we prove the following result, which immediately
implies Theorem \ref{thm: intro}:
\begin{thm}\name{thm: example}
Let $n \geq 9$ be an odd number divisible by 3, and let $\mathcal{P}$ be
an isosceles triangle with equal angles $1/n$. Let $\mathcal{Q}$ be
the parking garage made of four copies of $\mathcal{P}$ glued as in Figure
\ref{parking}, so that $\mathcal{Q}$ has vertices (in cyclic order) with angles $1/n, 2/n, 3/n,
(n-2)/n, 2/n, 3(n-2)/n$. Then $M_{\mathcal{P}}$ is a lattice surface and
$M_{\mathcal{Q}} \to M_{\mathcal{P}}$ is a translation map with one
aperiodic branch point. In particular $\mathcal{Q}$ is a non-lattice
parking garage with optimal dynamics. 
\end{thm}

\begin{proof}
The translation surface $M_{\mathcal{P}}$ is the double $n$-gon, one of
Veech's original examples of lattice surfaces \cite{Veech - dichotomy}. The groups 
$G_{\mathcal{P}}$ and $G_{\mathcal{Q}}$ are both equal to the dihedral
group $D_n$. Thus by Proposition \ref{prop: precisely},
the degree of the cover $M_{\mathcal{Q}} \to M_{\mathcal{P}}$ is
four. 
Again by Proposition
\ref{prop: precisely}, since $n$ is odd and divisible by 3, the only vertices which
correspond to branch points are the two vertices $z_1, z_2$ with angle $2/n$ (they
correspond to the case $k_i =2$ while the other vertices correspond to
$1$ or $3$). In the surface $M_{\mathcal{P}}$ there are two points
which correspond to vertices of equal angle in $\mathcal{P}$ (the
centers of the two $n$-gons), and these points are known to be
aperiodic \cite{HS}. We need to check that $z_1$ and $z_2$
both map to the same point in $M_{\mathcal{P}}$. This follows from the
fact that both are opposite the vertex $z_3$ with angle $3/n$, which also
corresponds to the center of an $n$-gon, so in $M_{\mathcal{P}}$
project to a point which is distinct from $z_3$. 
\end{proof}

\begin{remark}\name{remark: connection points}
As of this writing, it is not known whether the center of the regular
$n$-gon is a connection point on the double $n$-gon surface. If this
turns out to be the case for some $n$ which is an odd multiple of 3,
then by Corollary \ref{cor: suitable cover}(2), our 
construction satisfies strict ergodicity and topological dichotomy. See
\cite{AS} for some recent related results. 
\end{remark}

\begin{remark}\name{remark: quadratic growth}
Since our
examples are obtained by taking branched covers over lattice surfaces,
a theorem of Eskin, Marklof and Morris \cite[Thm. 8.12]{EMM} shows that our examples
also satisfy a quadratic growth estimate of the form
$N_{\mathcal{P}}(T) \sim cT^2$; moreover \S 9 of \cite{EMM} explains
how one may explicitly compute the constant $c$. 
\end{remark}

\section{Non-lattice optimal polygons are hard to find}\name{section: meital thesis} 
In this section we present results indicating that the above
considerations will not easily yield a non-lattice
polygon with optimal dynamics. 
Isolating the properties necessary for our proof of Theorem
\ref{thm: example}, we
say that polygons $\mathcal{P}, \mathcal{Q}$ are a {\em suitable pair} if the
following hold:
\begin{itemize}
\item
$\mathcal{P}$ is a lattice polygon. 
\item
$\mathcal{P}$ tiles $\mathcal{Q}$ by reflections. 
\item
The corresponding cover $M_{\mathcal{Q}} \to M_{\mathcal{P}}$ as in
Proposition \ref{prop: precisely} has a unique branch point which is aperiodic. 

\end{itemize}

In her M.Sc. thesis at Ben Gurion University, 
the first-named author conducted an extensive
 search for a suitable pair of polygons. By Corollary \ref{cor:
   suitable cover}, such a pair will have yielded a non-lattice
 polygon with optimal dynamics. The search begins with a list of
 candidates for $\mathcal{P},$ i.e. a list of currently known lattice
 polygons. At present, due to work of many authors, there is a fairly
 large list of known lattice 
 polygons but there is no classification of all lattice
 polygons. The following is the main result of \cite{Meital thesis}:

\begin{thm}[M. Cohen] \name{thm: Meital thesis}
Among all currently known lattice polygons, there is no $\mathcal{P}$
for which there is $\mathcal{Q}$ such that $\mathcal{P}, \mathcal{Q}$
is a suitable pair. 

\end{thm}

The proof of Theorem \ref{thm: Meital thesis} contains a detailed
case-by-case analysis for each of the different possible
$\mathcal{P}$. These cases involve some common arguments which we will
illustrate in this section, by proving the special case in which
$\mathcal{P}$ is any of the obtuse 
triangles investigated by Ward \cite{Ward}: 

\begin{thm} \name{thm: for paper}
Let $\mathcal{P}= \mathcal{P}_n$ be the (lattice) triangle with angles
$\displaystyle{\left(\frac1n, \frac{1}{2n}, \frac{2n-3}{2n}\right)}$. Then there
is no polygon $\mathcal{Q}$ 
for which $\mathcal{P}, \mathcal{Q}$ is a suitable pair.  
\end{thm}

Our proof relies on some auxilliary statements which are of independent
interest. In all of them, $M_{\mathcal{Q}} \to M_{\mathcal{P}}$ is the
branched cover with unique branch point corresponding to a suitable
pair $\mathcal{P}, \mathcal{Q}$. These statements are also valid in
the more general case in which $\mathcal{P, Q}$ are parking
garages. 

Recall that an affine automorphism of a translation surface is a
homeomorphism which is linear in
charts. We denote by $\Aff(M)$ the group of affine automorphisms of
$M$ and by $D: \Aff(M) \to \GL_2(\R)$ the homomorphism mapping an
affine automorphism to its linear part. Note that we allow
orientation-reversing affine automorphisms, i.e. $\det \varphi$ may be
1 or -1. 

We denote by 
$G_{\mathcal{Q}} \subset G_{\mathcal{P}}$ the corresponding reflection
groups. From the description of the unfolding
construction, it is clear that $G_{\mathcal{P}}$ acts on
$M_{\mathcal{P}}$ by permuting copies of $\mathcal{P}$, and this
yields an action by affine
automorphisms. Moreover, since $M_{\mathcal{Q}}$ is also made of
copies of $\mathcal{P}$, any $g \in G_{\mathcal{P}}$ maps
$M_{\mathcal{Q}}$ to a surface $gM_{\mathcal{Q}}$ with a covering map
$gM_{\mathcal{Q}} \to M_{\mathcal{P}}$. 

\begin{lem}\name{lem: group}
The branch point of the cover $M_{\mathcal{Q}} \to
M_{\mathcal{P}}$ is fixed by the subgroup 
$$\{g \in G_{\mathcal{P}} : gM_{\mathcal{Q}} = M_{\mathcal{Q}} \};$$
in particular it is fixed by $G_{\mathcal{Q}}$. 

\end{lem}
\begin{proof}
Any $g \in G_{\mathcal{P}}$ maps branch points of the cover $M_{\mathcal{Q}}
\to M_{\mathcal{P}}$ to branch points of $gM_{\mathcal{Q}} \to M_{\mathcal{P}}$; the
lemma follows from our assumption that there is a unique branch
point. 
\end{proof}

\begin{lem}\name{lem: fixed points}
If an affine automorphism $\varphi$ of a translation surface has infinitely many
fixed points then $D\varphi$ fixes a nonzero vector, in its linear
action on $\R^2$. 
\end{lem}

\begin{proof}
Suppose by contradiction that the linear action of $D\varphi$ on the
plane has zero
as a unique fixed point, and let $M_{\varphi}$ be the set of fixed points for
$\varphi$. For any $x \in M_{\varphi}$ which is not a singularity,
there is a chart from a neighborhood $U_x$ of $x$ to $\R^2$ with
$x \mapsto 0$, and a
smaller neighborhood $V_x \subset U_x$, such that $\varphi(V_x)\subset U_x$ and when
expressed in this chart, $\varphi|_{V_x}$ is given by the linear
action of $D\varphi$ on the plane. In particular $x$ is the only 
fixed point in $V_x$. Similarly, if $x \in M_{\varphi}$ is a
singularity, then there is a neighborhood $U_x$ of $x$ which maps to
$\R^2$ via a finite branched cover ramified at $x \mapsto 0$, such
that the action of $\varphi$ in $V_x 
\subset U_x$ covers the linear
action of $D\varphi$. Again we see that $x$ is the only fixed point in
$V_x$. By compactness we find that
$M_{\varphi}$ is finite, contrary to hypothesis. 
\end{proof}

\begin{lem}\name{lem: fixed implies periodic}
Suppose $M$ is a lattice surface and $\varphi \in \Aff(M)$ has $D
\varphi = -\mathrm{Id}$. Then a fixed point for  
$\varphi$ is periodic. 

\end{lem}

\begin{proof}
Let 
$$F_1 = \{\sigma \in \Aff(M) : D\sigma = - \mathrm{Id} \}
.$$ 
Then $\varphi \in F_1 $ and $F_1$ is finite, since it is a coset for
the group $\ker D$ which is known to be
finite. 
Let $\mathcal{A}\subset M$ be the set of points which are fixed by
some $\sigma \in F_1$. By Lemma \ref{lem: fixed points} this is a
finite set, which contains the fixed points for $\varphi$. 
Thus in order to prove the Lemma, it suffices to show that $
\mathcal{A}$ is  
$\Aff(M)$-invariant.

Let $\psi \in \Aff(M)$, and let $x \in \mathcal{A}$, so that $x =
\sigma(x)$ with $D \sigma = -\mathrm{Id}.$ Since -Id is central in
$\GL_2(\R)$, 
$D(\sigma \, \psi) = D (\psi \, \sigma)$, so there is $f \in \ker D$ such
that $\psi \, \sigma = f \, \sigma \, \psi$. Therefore
$$
\psi(x) = \psi \, \sigma (x)  = f \sigma \, \psi(x), \ \ \mathrm{and}
\ f\sigma \in F_1.
$$  
This proves that $\psi (x) \in \mathcal{A}$. 
\end{proof}

\begin{remark}
This improves Theorem 10 of \cite{GHS}, where a similar conclusion is
obtained under the additional assumptions that $M$ is hyperelliptic
and $\Aff(M)$ is generated
by elliptic elements. 

\end{remark}

The following are immediate consequences: 
\begin{cor}\name{cor: -id}
Suppose $\mathcal{P}, \mathcal{Q}$ is a suitable pair. Then 
\begin{itemize} 
\item
$-\mathrm{Id} \notin D(G_{\mathcal{Q}}).$
\item
None of the
angles between two edges of $\mathcal{Q}$ are of the form $p/q$ with
$\gcd(p,q)=1$ and $q$ even. 
\end{itemize}

\end{cor}

\begin{proof}[Proof of Theorem \ref{thm: for paper}]
We will suppose that $\mathcal{Q}$ is such that $\mathcal{P, Q}$ are a
suitable pair and reach a contradiction. If $n$ is even, then
$\Aff(M_{\mathcal{P}})$ contains a rotation by 
$\pi$ which fixes the points in $M_{\mathcal{P}}$ coming from vertices
of $\mathcal{P}$. Thus by Lemma \ref{lem: fixed implies periodic} all
vertices of $\mathcal{P}$ give rise to periodic points, contradicting
Proposition \ref{prop: precisely}(2). So $n$ must be odd.

	\begin{figure}[h!]%
		\begin{center}%
	\includegraphics[scale=0.65]{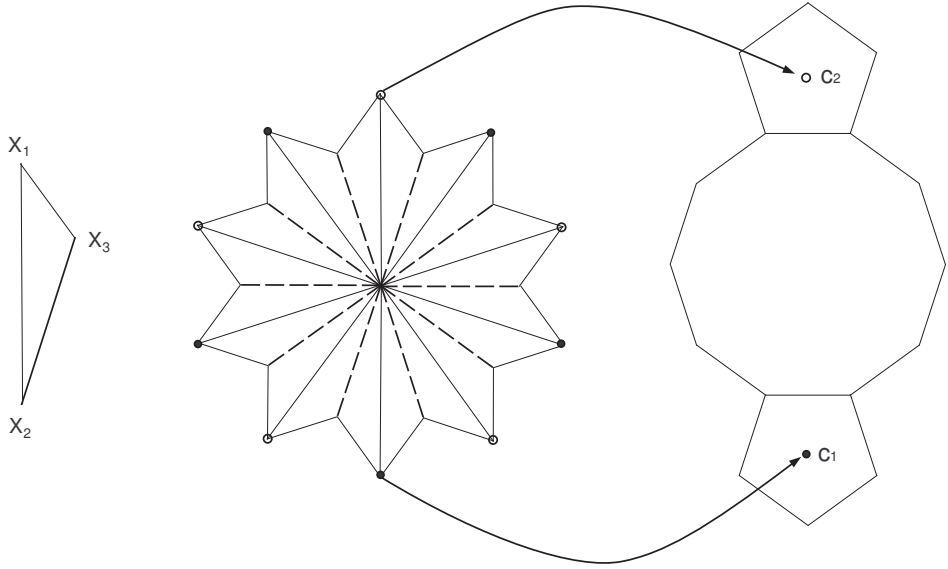}
		\caption{Ward's surface}
		\label{ward_surface}			
		\end{center}
		\end{figure}

Let $x_1, x_2, x_3$ be the vertices of $\mathcal{P}$ with
corresponding angles $1/n, 1/2n, (2n-3)/2n$. Then $x_3$ gives rise to a
singularity, hence a periodic 
point. Also using Lemma \ref{lem: fixed implies periodic} and the
rotation by $\pi$, one sees that $x_2$ also gives rise to a periodic
point. So the unique branch point must correspond to the vertex
$x_1$. The images of the vertex $x_1$ in $\mathcal{P}$ give rise to
two regular points in $M_{\mathcal{P}}$, marked $c_1, c_2$ in Figure
\ref{ward_surface}.
Any element
of $G_{\mathcal{P}}$ acts on $\{c_1, c_2\}$ by a permutation, so by Lemma 
\ref{lem: group}, $G_{\mathcal{Q}}$ must be contained in the subgroup
of index two fixing both of the $c_i$. Let $e_1$ be the edge of $\mathcal{P}$ opposite
$x_1$. Since the reflection in $e_1$, or any edge which is an image of
$e_1$ under $G_{\mathcal{P}}$, swaps the $c_i$, we have:
\eq{eq: restriction1}{e_1 \ \mathrm{\ is \ not \ an \ external \ edge
    \ of \ } \mathcal{Q}.}

We now claim that 
in $\mathcal{Q}$, any vertex which corresponds
to the vertex $x_3$ from $\mathcal{P}$ is always doubled,
i.e. consists of an angle of $(2n-3)/n$.
 Indeed, for any polygon
$\mathcal{P}_0$, the group $G_{\mathcal{P}_0}$ is the 
dihedral group $D_N$ where $N$ is the least common multiple of the
denominators of the angles at vertices of $\mathcal{P}_0$. In
particular it contains -Id when $N$ is even. 
Writing $(2n-3)/2n$ in reduced form we have an even denominator, and
since, by Corollary \ref{cor: -id}, $-\mathrm{Id} \notin
G_{\mathcal{Q}}$, in $\mathcal{Q}$ the angle at vertex $x_3$ must be
multiplied by an even integer $2k$. Since $2k (2n-3)/2n$ is bigger
than $2$ if $k>1$, and since the total angle at a vertex of a polygon
is at most $2\pi$, we must have $k=1$, i.e. any vertex in
$\mathcal{Q}$ corresponding to the vertex $x_3$ is always
doubled. This establishes the claim. {\em It is here that we have used
  the assumption that $\mathcal{Q}$ is a polygon and not a parking
  garage.} 

There are two possible configurations in which a vertex $x_3$ is
doubled, as shown in Figure \ref{ward_2_options}. The bold lines
indicate lines which are external, 
i.e. boundary edges of $\mathcal{Q}$. By \equ{eq: restriction1}, the
configuration on the right cannot occur.

	\begin{figure}[h!]%
		\begin{center}%
	\includegraphics[scale=0.55]{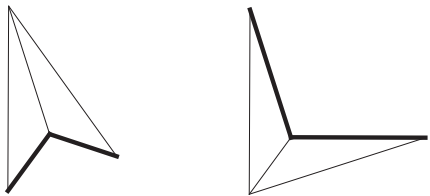}
		\caption{Two options to start the construction of $Q$}
		\label{ward_2_options}			
		\end{center}
		\end{figure}

Let us denote the polygon on the left hand side of Figure
\ref{ward_2_options} by $\mathcal{Q}_0$. It cannot be equal to
$\mathcal{Q}$, 
since it is a lattice polygon.
We now enlarge
$\mathcal{Q}_0$ by adding copies of
$\mathcal{P}$ step by step, as described in Figure
\ref{ward5_stages_Q}. Without loss of generality we first add triangle
number 1. By \equ{eq: restriction1}, the broken line indicates a side
which must be internal in $\mathcal{Q}$. Therefore, we add triangle number
2. We denote the resulting polygon by $\mathcal{Q}_1$. 
One can check by computing angles, using the fact that
$n$ is odd, and using Proposition
\ref{prop: precisely}(4) that the cover
$M_{\mathcal{Q}_1} \to M_{\mathcal{P}}$ will branch over
		the points $a$ corresponding to vertex $x_2$. Since the
                allowed branching is only over the points
                corresponding to $x_1$, we must have $\mathcal{Q}_1
                \subsetneq \mathcal{Q}$, so we continue the
                construction. Without loss of generality we add
                triangle number 3. Again, by \equ{eq: restriction1}, the
                broken line indicates a side which must be internal in
                $\mathcal{Q}$. Therefore, we add triangle number 4,
                obtaining $\mathcal{Q}_2$. 
		Now, using Proposition \ref{prop: precisely}(4) again,
                in the cover $M_{\mathcal{Q}_2} \to M_{\mathcal{P}}$ we have 
                branching over two vertices $u$ and $v$ which are both
                of type $x_1$ and correspond to distinct points $c_1$
                and $c_2$ 
                in $M_{\mathcal{P}}$. This implies $\mathcal{Q}_2
                \subsetneq \mathcal{Q}$. 
	\begin{figure}[h!]
	\begin{center}
	\includegraphics[scale=0.5]{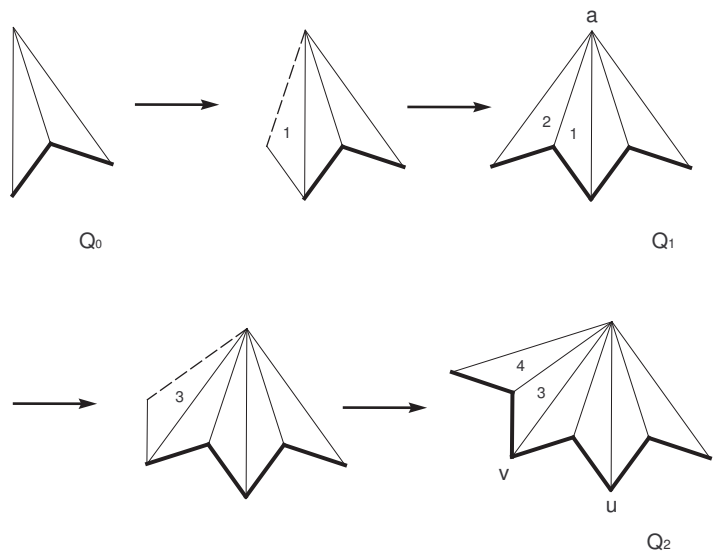}
	\caption{Steps of the construction of $Q$}
	\label{ward5_stages_Q}			
	\end{center}
	\end{figure}

		Since both vertices $u$ and $v$ are delimited by 2
                external sides, we cannot change the angle to prevent the
                branching over one of these 
		points. This means that no matter how we continue to
                construct $\mathcal{Q}$, the branching in the cover
                $M_{\mathcal{Q}} \to M_{\mathcal{P}}$ will occur over
                at least two points -- a contradiction. 
\end{proof}

\end{document}